\documentclass{amsart}

\usepackage{newtxtext,newtxmath}
\usepackage[papersize={17cm,24cm},text={12.5cm,19.5cm},centering]{geometry}

\usepackage[initials,lite]{amsrefs}

\newtheorem{theorem}{Theorem}
\newtheorem{proposition}{Proposition}

\theoremstyle{definition}
\newtheorem{definition}{Definition}

\newcommand{\CA}{\mathcal{A}}
\newcommand{\g}{\mathfrak{G}}
\newcommand{\h}{\mathfrak{H}}
\newcommand{\dbar}{\bar{\partial}}
\newcommand{\Linf}{$L_\infty$}
\newcommand{\bull}{\bullet}
\newcommand{\eps}{\varepsilon}

\DeclareMathOperator{\MC}{\mathsf{MC}}
\DeclareMathOperator{\ad}{ad}
\DeclareMathOperator{\End}{End}

\begin{document}

\title[Maurer--Cartan elements and homotopical perturbation
theory]{Maurer--Cartan Elements and Homotopical Perturbation Theory}

\author{Ezra Getzler}

\address{Department of Mathematics, Northwestern University, Evanston,  Illinois, USA}

\email{getzler@northwestern.edu}

\thanks{This paper was written while the author was partially
  supported by a Fellowship in Mathematics of the Simons Foundation.}

\begin{abstract}
  Let $L$ be a (pro-nilpotent) curved \Linf-algebra, and let
  $h:L\to L[-1]$ be a homotopy between $L$ and a subcomplex $M$. Using
  homotopical perturbation theory, Fukaya constructed from this data a
  curved \Linf-structure on $M$. We prove that projection from $L$ to
  $M$ induces a bijection between the set of Maurer--Cartan elements
  $x$ of $L$ such that $hx=0$ and the set of Maurer--Cartan elements
  of $M$.
\end{abstract}

\maketitle

Homological perturbation theory is concerned with the following
situation, which we call a \textbf{context}: a pair of complete
filtered cochain complexes $(V^\bull,\delta)$ and $(W^\bull,d)$,
filtered morphisms of complexes $f:W\to V$ and $g:V\to W$ such that
$gf=1_W$, and a map $h:V\to V[-1]$, compatible with the filtration,
such that
\begin{equation*}
  1_V = fg + \delta h + h\delta .
\end{equation*}
In addition, we assume the \textbf{side conditions}
\begin{equation*}
  h^2 = hf = gh = 0 .
\end{equation*}
For a review of this subject and its history, see Gugenheim and Lambe
\cite{GL}.

If $\mu:V\to V[1]$ is a deformation of the differential on $V$, in
the sense that $\delta_\mu^2=0$, where $\delta_\mu=\delta+\mu$, and if
$\mu$ has strictly positive filtration degree, then $\mu$ induces a
deformation of the above context, with
\begin{align*}
  h_\mu &= \sum_{n=0}^\infty (-h\mu)^nh &
  d_\mu &= d + \sum_{n=0}^\infty g(-\mu h)^n\mu f \\
  f_\mu &= \sum_{n=0}^\infty (-h\mu)^nf &
  g_\mu &= \sum_{n=0}^\infty g(-\mu h)^n .
\end{align*}
These expressions are convergent, by the hypothesis that $\mu$ has
strictly positive filtration degree and that the filtrations on $V$
and $W$ are complete.

Homotopical perturbation theory considers a more general perturbation,
in which $\mu$ is not only a perturbation of the differential on $V$,
but deforms $V$ to a non-trivial algebraic structure, such as an
associative algebra, commutative algebra, or Lie algebra. The outcome
is a \textbf{homotopy} algebraic structure of the same type on $W$,
and a morphism $F_\mu$ of homotopy algebras from $W$, with the
transported structure, to $V$. For associative, respectively
commutative and Lie algebras, a homotopy algebra is known as an
$A_\infty$-algebra, $C_\infty$-algebra and $L_\infty$-algebra
respectively. This theory was pioneered by Kadeishvili
\cite{Kadeishvili}, in the special case of associative algebras, and
under the assumption that the differential $d$ on $V$ and the
deformation of the differential $\delta$ on $W$ vanish. The case of
Lie algebras was taken up by Fukaya \cite{Fukaya}: he generalizes the
construction considerably, permitting the deformation on $V$ to be a
\textbf{curved \Linf-algebra}. For an overview of homotopical
perturbation theory in this context and others, see also Bandiera
\cite{Bandiera}, Berglund \cite{Berglund} and Loday--Vallette
\cite{LV}.

In this paper, we work in Fukaya's setting. A complete filtered
complex is a complex $(L,\delta)$ with a decreasing filtration of
finite length
\begin{equation*}
  L = F_0L \supset F_2L \supset F_3L \supset \dots
\end{equation*}
which is a complete metric space with respect to the metric
\begin{equation*}
  d_c(x,y) = \inf \{ c^{-k} \mid x-y \in F_kL \} .
\end{equation*}
Here, $c$ may be any real number greater than $1$.

A curved \Linf-algebra $\g^\bull$ is a graded vector space with a
complete filtration, together with a sequence of graded antisymmetric
brackets
\begin{equation*}
  [x_1,\dots,x_n]_n :\g^{i_1}\times \g^{i_n}\to \g^{i_1+\dots+i_n+2-n}
\end{equation*}
of degree $2-n$, $n\ge0$, satisfying certain relations which we will
recall below. Here,
\begin{equation*}
  [x_1,x_2]=[x_1,x_2]_2:\g^i\times \g^j\to \g^{i+j} ,
\end{equation*}
generalizes the bracket of graded Lie algebras,
\begin{equation*}
  \delta x_1=[x_1]_1:\g^i\to \g^{i+1}
\end{equation*}
is an operator analogous to a differential, and $R=[~]_0\in \g^2$ is
the curvature of $\delta$, in the sense that
\begin{equation*}
  \delta^2x = [R,x] .
\end{equation*}
Following Fukaya, we assume that $R\in F_1\g^2$ is an element of
strictly positive filtration degree.

A Maurer--Cartan element of a curved \Linf-algebra is an element $x$
of $F_1\g^1$ satisfying the equation
\begin{equation*}
  \sum_{n=0}^\infty \frac{1}{n!} \, [x,\dots,x]_n = 0 .
\end{equation*}
The sum makes sense because the filtration degree of the terms
$[x,\dots,x]_n$ converges to infinity with $n$, owing to the
hypothesis that the filtration degree of $x$ is strictly
positive. Denote by $\MC(\g)$ the set of Maurer-Cartan elements of
$\g$.

Now suppose that $\g$ is endowed in addition with a homological
perturbation theory context
\begin{equation*}
  f:(\h,d)\rightleftarrows(\g,\delta):g ,
\end{equation*}
with homotopy $h:\g\to\g[-1]$. Fukaya associates to these data a
curved \Linf-structure on $\h$. Consider the sets $\MC(\h)$ and
\begin{equation*}
  \MC(\g,h) = \{ x \in \MC(\g) \mid hx=0 \} .
\end{equation*}
We call $\MC(\g,h)$ the Kuranishi set of $(\g,h)$. The goal of this
paper is to give a self-contained proof of the following result.
\begin{theorem}
  \label{main}
  The morphism $g$ induces a bijection from $\MC(\g,h)$ to $\MC(\h)$.
\end{theorem}
The \Linf-morphism from $\h$ to $\g$ constructed by Fukaya induces a
map from $\MC(\h)$ to $\MC(\g)$, whose image is actually seen by
inspection to lie in $\MC(\g,h)$. Thus, our task will be to show that
$g$, on restriction to $\MC(\g,h)$, gives an inverse to this map.

In \cite{Kuranishi}, Kuranishi considers the following analogue of the
above situation. The differential graded Lie algebra $\g^\bull$ is the
Dolbeault resolution $\CA^{0,\bull}(X,T)$ of the sheaf of holomorphic
vector fields on a compact complex manifold $X$, the complex
$\h^\bull$ is the subspace of harmonic forms
\begin{equation*}
  \mathcal{H}^{0,\bull}(X,T) \subset \CA^{0,\bull}(X,T)
\end{equation*}
with respect to a choice of Hermitian metric on $X$, and has vanishing
differential, $f$ is the inclusion, $g$ is the orthogonal projection
with respect to the $L^2$-inner product, and $h$ is the operator
\begin{equation*}
  h = \bigl(\dbar^*\dbar+\dbar\dbar^*\bigr)^{-1} \dbar^* .
\end{equation*}
We have neglected the Banach completions which are needed in order to
make sense of the infinite sums in the formulas, and are considerably
more complicated to describe than the non-Archimedean Banach spaces,
or complete filtered vector spaces, considered in this paper. For
example, one may take the completion of $\CA^{0,i}(X,T)$ consisting of
differential forms whose coefficients are in the Sobolev space of
functions with square-integrable derivatives up to order
$\dim_{\mathbb{C}}(X)-i+1$.

In this setting, $\MC(\g)$ is the space of solutions of the
Kodaira--Spencer equation on $\CA^{0,1}(X,T)$,
\begin{equation*}
  \dbar x + \tfrac{1}{2} [x,x] = 0 ,
\end{equation*}
and $\MC(\g,h)$ is the space of solutions of the Kodaira--Spencer
equation that also satisfy the Kuranishi gauge condition
\begin{equation*}
  \dbar^*x = 0 .
\end{equation*}
Kuranishi shows that this subset is a slice for the foliation of
$\MC(\g)$ induced by the action of $\Gamma(X,T)=\CA^{0,0}(X,T)$ on
$\CA^{0,1}(X,T)$: the tangent space to the leaf through $x\in\MC(\g)$
is the subspace
\begin{equation*}
  \bigl\{ \dbar z + [x,z] \mid z \in \Gamma(X,T) \bigr\} .
\end{equation*}
In the algebraic setting, this was proved (and considerably
generalized) in \cite{linf}, in the case where the filtration on $L$
is finite, and extended to the case of complete filtrations by
Dolgushev and Rogers \cite{DR}, with essentially the same proof.

We now turn to the proof of Theorem~\ref{main}. For technical reasons,
we shift the degrees of our curved \Linf-algebras down by one: this
has the effect of giving all of the brackets $[x_1,\dots,x_n]_n$
degree one, and considerably simplifies the signs arising in the
formulas. Introduce the shifted brackets on $L=\g[1]$: if
$x_j\in L^{i_j}$ and $sx_j\in\g^{i_j+1}$ correspond to each other
under the identification of $\g$ as the suspension of $L$, then
\begin{equation*}
  \lambda_n(x_1,\dots,x_n) = (-1)^{\sum_{i=1}^n (n-i)|x_i|} \,
  [sx_1,\dots,sx_n]_n .
\end{equation*}
The shifted brackets on $L$ are graded symmetric, and in terms of
them, the Maurer--Cartan equation on $F_1L^0$ becomes
\begin{equation*}
  \sum_{n=0}^\infty \frac{1}{n!} \, \lambda_n(x,\dots,x) = 0 .
\end{equation*}

Given a pair $L$ and $M$ of complete filtered complexes, we consider
the complete filtered complex $S^{n,*}(L,M)$, where
\begin{equation*}
  S^{n,i}(L,M)
  = \{ \text{filtered graded symmetric $n$-linear maps from
    $L$ to $M$ of degree $i$} \} .
\end{equation*}
Here, an $n$-linear map $a_n$ is graded symmetric if for all
$1\le j<n$,
\begin{equation*}
  a_n(x_1,\dots,x_{j+1},x_j,\dots,x_n) = (-1)^{|x_j|\,|x_{j+1}|} \,
  a_n(x_1,\dots,x_j,x_{j+1},\dots,x_n) ,
\end{equation*}
filtered if
\begin{equation*}
  a_n(F_{k_1}L,\dots,F_{k_n}L) \subset F_{k_1+\dots+k_n}M ,
\end{equation*}
and has degree $i$ if
\begin{equation*}
  |a_n(x_1,\dots,x_n)| = |x_1| + \dots + |x_n| + i .
\end{equation*}
Let $S^*(L,M)$ be the complex of inhomogeneous multilinear maps
\begin{equation*}
  S^i(L,M) = F_1M \times \prod_{n=1}^\infty S^{n,i}(L,M) \subset
  \prod_{n=0}^\infty S^{n,i}(L,M) ,
\end{equation*}
with filtration
\begin{equation*}
  F_kS^i(L,M) = \{ (a_0,a_1,\dots) \in S^i(L,M) \mid
  a_n(F_{k_1}L,\dots,F_{k_n}L) \subset F_{k_1+\dots+k_n+k}M \} .
\end{equation*}

There is a binary operation from $S^i(L,M)\times S^j(L,L)$ to
$S^{i+j}(L,M)$, defined by the formula
\begin{equation*}
  (a\circ b)_n(x_1,\dots,x_n) = \sum_{\sigma\in S_n} (-1)^\eps \sum_{k=0}^n
  \frac{1}{k!(n-k)!} \, a_{n-k+1}(b_k(x_{\sigma_1},\dots),\dots,x_{\sigma_n}) .
\end{equation*}
For example,
\begin{equation*}
  (a\circ b)_0 = a_1(b_0)
\end{equation*}
and
\begin{equation*}
  (a\circ b)_1(x) = a_2(b_0,x) + a_1(b_1(x)) .
\end{equation*}
This product satisfies Gerstenhaber's pre-Lie algebra axiom:
\begin{equation}
  \label{pre-Lie}
  (a\circ b)\circ c - (-1)^{|b|\,|c|} \, (a\circ b)\circ c
  = a\circ(b\circ c) - (-1)^{|b|\,|c|} \, a\circ(b\circ c) .
\end{equation}

If $\lambda\in S^1(L,L)$ is an element of degree $1$, consider the
operation
\begin{equation*}
  \delta_\lambda f = \delta f + \lambda\circ f - (-1)^{|f|} \, f\circ\lambda .
\end{equation*}
By \eqref{pre-Lie}, we have
\begin{equation*}
  \delta^2 f = (\delta\lambda+\lambda\circ\lambda)\circ f -
  f\circ(d\lambda+\lambda\circ\lambda) .
\end{equation*}

\begin{definition}
  A curved \Linf-algebra is a filtered complex $L$ together with an
  element $\lambda\in S^1(L,L)$ such that
  $\delta\lambda+\lambda\circ\lambda=0$. It is an \Linf-algebra if its
  curvature $\lambda_0$ vanishes. It is \textbf{pro-nilpotent} if
  $\lambda\in F_1S^1(L,L)$.
\end{definition}

A differential graded Lie algebra $\g^*$ gives rise to an
\Linf-algebra by setting $L=\g[1]$, with the discrete filtration
$F_0L=L$ and $F_1L=0$, and with
\begin{equation*}
  \lambda_n(x_1,\dots,x_n) =
  \begin{cases}
    (-1)^{|x_1|} \, [x_1,x_2] , & n=2 , \\
    0 , & n\ne2 .
  \end{cases}
\end{equation*}
For example, if $E$ is a complex of vector bundles on a manifold $M$
with differential $\delta$, and $D$ is a flat connection on $E$,
preserving degree, such that the covariant derivative of $\delta$
vanishes, then the total complex of the bicomplex
$\Omega^*(M,\End(E))$ of differential forms on $M$ with values in the
bundle of graded algebras $\End(E)$ and differentials $\ad(D)$ and
$\ad(\delta)$ is a differential graded Lie algebra, with bracket equal
to the graded commutator.

On the other hand, if the connection $D$ is not flat, but has
curvature
\begin{equation*}
  R \in \Omega^2(M,\End(E)) ,
\end{equation*}
then $\Omega^*(M,\End(E))$ is a curved differential graded Lie
algebra, with curvature $E$.

Curved \Linf-algebras are a common generalization of \Linf-algebras
and curved differential graded Lie algebras.

Given $a\in S^i(L,M)$ and $b\in S^0(K,L)$, define the composition
$a\bull b\in S^i(K,M)$ by the formula
\begin{multline*}
  (a\bull b)_n(x_1,\dots,x_n) = \sum_{\sigma\in S_n} (-1)^\eps \\
  \sum_{k=0}^n \frac{1}{k!} \sum_{n_1+\dots+n_k=n} \frac{1}{n_1!\dots n_k!} \,
  a_k(b_{n_1}(x_{\sigma_1},\dots),\dots,b_{n_k}(\dots,x_{\sigma_n})) .
\end{multline*}
It is in order for this operation be well-defined that we have imposed
the restriction that $b_0\in F_1L$.

The operation $a\bull b$ is associative: if $a\in S^i(L,M)$,
$b\in S^0(K,L)$ and $c\in S^0(J,K)$, then
\begin{equation*}
  (a\bull b)\bull c = a\bull(b\bull c) \in S^i(J,M) .
\end{equation*}
It has $1_L\in S^0(L,L)$ as a right inverse
\begin{equation*}
  a\bull 1_L = a ,
\end{equation*}
and $1_M$ as a left inverse
\begin{equation*}
  1_M\bull a = a .
\end{equation*}

We generalize the product $a\circ b$ as follows. Extend the ground
ring to the exterior algebra with a generator $\eps$ of degree
$-1$. If $a\in S^i(L,M)$, $b\in S^0(K,L)$ and $\beta\in S^1(K,L)$,
define $a\circ_b\beta$ by the formula
\begin{equation*}
  a\bull(b+\beta\,\eps) = a\bull b + (a\circ_b\beta)\,\eps .
\end{equation*}
In particular, if $K=L$, we have
\begin{equation*}
  a\bull(1_L+\beta\,\eps) = a + (a\circ\beta)\,\eps .
\end{equation*}
If $c\in S^0(J,K)$, then by the associativity of $\bull$, we have
\begin{equation*}
  (a\circ_b\beta)\bull c = a\circ_{b\bull c}(\beta\bull c) .
\end{equation*}
Also, we have
\begin{equation*}
  \delta(a\bull b) = (\delta a)\bull b + (-1)^{|a|} \, a\circ_b(\delta b) .
\end{equation*}

\begin{definition}
  Let $(L,\lambda)$ and $(M,\mu)$ be curved \Linf-algebras. An
  \textbf{\Linf-morphism} $\phi:L\to M$ (sometimes called a shmap)
  is an element $\phi\in S^0(L,M)$ satisfying the equation
  \begin{equation*}
    \delta\phi + \mu\bull\phi = \phi\circ\lambda .
  \end{equation*}
\end{definition}

\begin{proposition}
  The composition $\psi\bull\phi\in S^0(K,M)$ of two \Linf-morphisms
  $\phi\in S^0(K,L)$ and $\psi\in S^0(L,M)$ is an \Linf-morphism.
\end{proposition}
\begin{proof}
  We argue as follows:
  \begin{align*}
    \delta(\psi\bull\phi) + \mu\bull(\psi\bull\phi)
    &= (\delta\psi+\mu\bull\psi)\bull\phi +
      \psi\circ_\phi\delta\phi \\
    &= (\psi\circ\lambda)\bull\phi + \psi\circ_\phi\delta\phi \\
    &= \psi\circ_\phi(\delta\phi+\lambda\bull\phi) \\
    &= \psi\circ_\phi(\phi\circ\kappa) = (\psi\bull\phi)\circ\kappa .
      \qedhere
  \end{align*}
\end{proof}

Denote the category of curved \Linf-algebras and \Linf-morphisms by
$L_\infty$. The curved \Linf-algebra $0$ is a terminal object of this
category, and a point $x\in L_\infty(0,L)$ of a curved \Linf-algebra
is called a Maurer--Cartan element: in other words, $x$ is an element
of $F_1L$ such that
\begin{equation*}
  \lambda\bull x = \sum_{n=0}^\infty \, \frac{1}{n!} \,
  \lambda_n(x,\dots,x) = 0 .
\end{equation*}
The Maurer--Cartan set $\MC(L)$ is the set of all Maurer--Cartan
elements, in other words, the set of points $L_\infty(0,L)$ of $L$.
We see that $\MC$ is a functor from the category of curved
\Linf-algebras to the category of sets.

When $F_1L^0$ is finite-dimensional, $\MC(L)$ is an algebraic
subvariety of the affine space $F_1L^0$. When $L^0$ and $M^0$ are
finite-dimensional and $\phi:L\to M$ is an \Linf-morphism,
$\MC(\phi):\MC(L)\to\MC(M)$ is an algebraic morphism of affine
varieties.

Now suppose that $L$ is endowed with a homological perturbation theory
context
\begin{equation*}
  f:(M,d)\rightleftarrows(L,\delta):g ,
\end{equation*}
with homotopy $h:L\to L[-1]$. Given a curved \Linf-structure $\lambda$
on $L$, Fukaya defines a curved \Linf-structure $\mu$ on $M$, and an
\Linf-morphism $F:L\to M$ such that $gF=1_M$. The formulas for $\mu$
and $F$ are determined by the solution of a fixed-point problem: in
order for this solution to exist, we must assume that $\lambda$ is
pro-nilpotent.
\begin{theorem}[Fukaya]
  There is a unique solution in $S^0(M,L)$ of the fixed-point equation
  \begin{equation}
    \label{F}
    F = f - h\lambda\bull F .
  \end{equation}
  Furthermore, $\mu=g\lambda\bull F\in S^1(M,M)$ is a curved
  \Linf-structure on $M$, and $F$ is an \Linf-map from $(M,d,\mu)$ to
  $(L,\delta,\lambda)$.
\end{theorem}
\begin{proof}
  The existence and uniqueness of $F$ follows from the hypothesis of
  pro-nilpotence of $\lambda$, since the map $F\mapsto f-h\lambda F$
  is a contraction mapping.

  It remains to show the vanishing of the quantities
  \begin{align*}
    \alpha
    &= \delta\bull F + \lambda\bull F - F\circ\delta - F\circ\mu \in
      S^1(M,L) \intertext{and}
    \beta
    &= \delta\mu + \mu\circ\mu \in S^1(M,M) .
  \end{align*}
  This follows from the equations
  \begin{equation*}
    \begin{cases}
      \alpha = - h(\lambda\circ_F\alpha) & \\
      \beta = - g(\lambda\circ_F\alpha) &
    \end{cases}
  \end{equation*}
  which may be proved by explicit calculation. Applying the
  contraction mapping theorem once more, we see that $\alpha=0$ is the
  unique solution of the first equation, from which we conclude from
  the second equation that $\beta$ vanishes as well.
\end{proof}

It is now easy to prove Theorem~\ref{main}. Applying $g$ to both sides
of \eqref{F}, we see that $gF=1_M$. It is also clear that if $x$ is a
Maurer--Cartan element of $M$, that the Maurer--Cartan element
$F(x)\in\MC(L)$ satisfies the gauge condition $hF(x)=0$:
\begin{equation*}
  hF(x) = hf(x) - h^2\lambda\bull F(x)
\end{equation*}
vanishes by the side conditions $hf=0$ and $h^2=0$. It remains to
verify that if $x\in\MC(L,h)$, then $F(gx)=x$. In fact, an explicit
calculation shows that if $x$ is an element of $F_1L^0$ satisfying
$hx=0$, then
\begin{equation*}
  x - F(gx) = h\lambda(F(gx)) - h\lambda(x) .
\end{equation*}
The vanishing of $x-F(gx)$ follows by the contraction mapping theorem,
since the map $y\mapsto-h\lambda(y)$ is a contraction on $L^0$, by the
mean value theorem. To see that this map is a contraction, it suffices
to observe that its partial derivatives raise filtration degree at
every point.


\begin{bibdiv}
  \begin{biblist}

\bib{Bandiera}{thesis}{
   author={Bandiera, Ruggiero},
   title={Higher Deligne groupoids, derived brackets and deformation
     problems in holomorphic Poisson geometry},
   organization={University of Rome La Sapienza},
   date={January 2015},
   eprint={
http://www1.mat.uniroma1.it/ricerca/dottorato/TESI/ARCHIVIO/bandierarugero.pdf}
}

\bib{Berglund}{article}{
   author={Berglund, Alexander},
   title={Homological perturbation theory for algebras over operads},
   journal={Algebr. Geom. Topol.},
   volume={14},
   date={2014},
   number={5},
   pages={2511--2548},
}

\bib{DR}{article}{
   author={Dolgushev, Vasily A.},
   author={Rogers, Christopher L.},
   title={A version of the Goldman-Millson theorem for filtered
   $L_\infty$-algebras},
   journal={J. Algebra},
   volume={430},
   date={2015},
   pages={260--302},
}

\bib{Fukaya}{article}{
   author={Fukaya, Kenji},
   title={Deformation theory, homological algebra and mirror symmetry},
   conference={
      title={Geometry and physics of branes},
      address={Como},
      date={2001},
   },
   book={
      series={Ser. High Energy Phys. Cosmol. Gravit.},
      publisher={IOP, Bristol},
   },
   date={2003},
   pages={121--209},
}

\bib{linf}{article}{
   author={Getzler, Ezra},
   title={Lie theory for nilpotent $L_\infty$-algebras},
   journal={Ann. of Math. (2)},
   volume={170},
   date={2009},
   number={1},
   pages={271--301},
}

\bib{GL}{article}{
   author={Gugenheim, V. K. A. M.},
   author={Lambe, L. A.},
   title={Perturbation theory in differential homological algebra. I},
   journal={Illinois J. Math.},
   volume={33},
   date={1989},
   number={4},
   pages={566--582},
}

\bib{Kadeishvili}{article}{
   author={Kadeishvili, T. V.},
   title={The algebraic structure in the homology of an $A(\infty
   )$-algebra},
   language={Russian, with English and Georgian summaries},
   journal={Soobshch. Akad. Nauk Gruzin. SSR},
   volume={108},
   date={1982},
   number={2},
   pages={249--252 (1983)},
}

\bib{Kuranishi}{article}{
   author={Kuranishi, M.},
   title={On the locally complete families of complex analytic structures},
   journal={Ann. of Math. (2)},
   volume={75},
   date={1962},
   pages={536--577},
}

\bib{LV}{book}{
   author={Loday, Jean-Louis},
   author={Vallette, Bruno},
   title={Algebraic operads},
   series={Grundlehren der Mathematischen Wissenschaften [Fundamental
   Principles of Mathematical Sciences]},
   volume={346},
   publisher={Springer, Heidelberg},
   date={2012},
}

  \end{biblist}
\end{bibdiv}
\end{document}